\pdfoutput=1
\documentclass{amsart}
\usepackage[utf8]{inputenc}

\usepackage[margin=1in]{geometry}
\usepackage[expansion=false]{microtype}

\usepackage{amsmath, amsfonts, amssymb, amsthm, amsopn, amscd}
\usepackage{eucal}

\usepackage{booktabs,enumitem}
\setlist[enumerate]{label=(\roman*),font=\normalfont}

\usepackage{tikz}
\usetikzlibrary{matrix}
\usepackage{tikz-cd}

\usepackage[pdfusetitle,unicode,bookmarksopen]{hyperref}

\numberwithin{equation}{section}

\theoremstyle{plain}
\newtheorem*{theorem*}{Theorem}
\newtheorem{theorem}[equation]{Theorem}
\newtheorem{proposition}[equation]{Proposition}
\newtheorem{lemma}[equation]{Lemma}
\newtheorem{corollary}[equation]{Corollary}

\theoremstyle{definition}
\newtheorem{definition}[equation]{Definition}
\newtheorem{construction}[equation]{Construction}

\newtheorem{remark}[equation]{Remark}

\let\scr=\mathcal
\let\bb=\mathbb
\let\rm=\mathrm

\def\Z{\bb Z}

\def\Q{\bb Q}

\def\F{\bb F}
\def\A{\bb A}
\def\P{\bb P}
\def\V{\bb V}
\def\1{\mathbf 1}
\def\h{\mathrm h}

\def\G{\mathbb G}
\makeatletter
\def\pt{{\mathpalette\pt@{.75}}}
\def\pt@#1#2{\mathord{\scalebox{#2}{$\m@th#1\bullet$}}}
\makeatother

\def\L{\mathrm{L}{}}

\let\into=\hookrightarrow
\let\onto=\twoheadrightarrow
\def\simto{\xrightarrow{\sim}}

\def\suchthat{\:\vert\:}

\DeclareMathOperator{\Sym}{Sym}

\def\id{\mathrm{id}}

\def\Map{\mathrm{Map}}

\def\Th{\mathrm{Th}}
\def\Pic{\mathrm{Pic}}
\def\GW{\mathrm{GW}}

\def\HH{\mathrm{H}}
\let\sect=\S
\def\S{\mathrm{S}}
\def\E{\mathbb{E}}

\DeclareMathOperator{\fib}{fib}

\def\Nis{\mathrm{Nis}}
\def\Zar{\mathrm{Zar}}

\def\et{\mathrm{\acute et}}

\def\Bl{\mathrm{Bl}}

\let\cat=\mathrm

\def\MGL{\mathrm{MGL}}

\def\KGL{\mathrm{KGL}}

\def\K{\mathrm{K}{}}

\def\SL{\mathrm{SL}}

\def\Mod{\cat{M}\mathrm{od}{}}

\def\Sm{{\cat{S}\mathrm{m}}}

\def\Vect{\cat{V}\mathrm{ect}{}}

\def\MS{\mathrm{MS}}

\def\Sp{\mathrm{Sp}}

\def\CAlg{\mathrm{CAlg}{}}

\def\MW{\mathrm{MW}}
\def\tK{\bar{\mathrm{K}}{}}

\def\GL{\mathrm{GL}}

\def\Mon{\mathrm{Mon}}

\def\sbu{\mathrm{sbu}}

\def\gys{\mathrm{gys}{}}
\def\m{\mathrm{m}}

\usepackage{relsize}
\usepackage[bbgreekl]{mathbbol}
\DeclareSymbolFontAlphabet{\mathbb}{AMSb} 
\DeclareSymbolFontAlphabet{\mathbbl}{bbold}

\let\lim=\relax
\DeclareMathOperator*{\lim}{lim}

\DeclareMathOperator*{\colim}{colim}

\let\phi=\varphi
\let\epsilon=\varepsilon

\title{Remarks on the motivic sphere without \texorpdfstring{$\mathbb A^1$}{A¹}-invariance}
\date{\today}

\author{Marc Hoyois}
\address{Fakultät für Mathematik\\
Universität Regensburg\\
Universitätsstr. 31\\
93040 Regensburg\\
Germany}
\email{\href{mailto:marc.hoyois@ur.de}{marc.hoyois@ur.de}}
\urladdr{\url{https://hoyois.app.uni-regensburg.de}}
\thanks{M.H.\ was partially supported by the Collaborative Research Center SFB 1085 \emph{Higher Invariants} funded by the DFG}

\begin{document}

\begin{abstract}
	We generalize several basic facts about the motivic sphere spectrum in $\A^1$-homotopy theory to the category $\MS$ of non-$\A^1$-invariant motivic spectra over a derived scheme.
	On the one hand, we show that all the Milnor–Witt K-theory relations hold in the graded endomorphism ring of the motivic sphere.
	On the other hand, we show that the positive eigenspace $\1_\Q^+$ of the rational motivic sphere is the rational motivic cohomology spectrum $\HH\Q$, which represents the eigenspaces of the Adams operations on rational algebraic K-theory.
	We deduce several familiar characterizations of $\HH\Q$-modules in $\MS$: a rational motivic spectrum is an $\HH\Q$-module iff it is orientable, iff the involution $\langle -1\rangle$ is the identity, iff the Hopf map $\eta$ is zero, iff it satisfies étale descent. Moreover, these conditions are automatic in many cases, for example over non-orderable fields and over $\Z[\zeta_n]$ for any $n\geq 3$.
\end{abstract}

\maketitle
\tableofcontents

\section{Introduction}

Let $S$ be a derived scheme. In this article, we prove several results about 
the motivic sphere spectrum over $S$, by which we mean
the unit object $\1$ of the symmetric monoidal category of motivic spectra $\MS_S$ introduced and studied in \cite{AnnalaIwasa2,AHI,atiyah}.
These refine familiar results about the $\A^1$-localization $\L_{\A^1}(\1)$ of $\1$, due for the most part to F.\ Morel.

Aside from the fact that, by design, $\A^1$-invariance does not hold in $\MS_S$, a key difference with stable $\A^1$-homotopy theory is that the multiplicative group $\G_\m$ is not a ``motivic sphere'' anymore, i.e., it is not $\otimes$-invertible in $\MS_S$. Instead, the Tate circle in $\MS_S$ is the desuspended projective line $\Sigma^{-1}\P^1$, which is a canonical direct summand of $\G_\m$ (whose complement is $\A^1$-contractible).

We first define various elements in the graded ring $\pi_0\Map(\1, (\Sigma^{-1}\P^1)^{\otimes *})$, which lift the elements of the same name in $\A^1$-homotopy theory: 
\begin{itemize}
	\item The Hopf map $\eta$ in degree $-1$ is induced by the multiplication of $\G_\m$ and is also the connecting map of the cofiber sequence $\P^1\to \P^2\to \P^2/\P^1$.
	\item The element $\epsilon$ in degree $0$ is the swap automorphism of $(\Sigma^{-1}\P^1)^{\otimes 2}$.
	\item For any unit $a\in\scr O(S)^\times$, the element $[a]$ in degree $1$ is induced by the point $a\in \G_\m$.
	\item For any unit $a\in\scr O(S)^\times$, the element $\langle a\rangle$ in degree $0$ is induced by the multiplication by $a$ on $\P^1$.
\end{itemize}
These elements are related by $\langle a\rangle=\eta[a]+1$ and $\epsilon=-\langle -1\rangle$.

\begin{theorem}[The Milnor–Witt K-theory relations, {Theorems~\ref{thm:easy-MW} and \ref{thm:steinberg}}]
	The following relations hold in the graded ring $\pi_0\Map(\1,(\Sigma^{-1}\P^1)^{\otimes *})$:
	\begin{enumerate}
		\item $[a][1-a]=0$ for all $a\in\scr O(S)^\times$ with $1-a\in\scr O(S)^\times$.
		\item $[ab]=[a]+[b]+\eta[a][b]$ for all $a,b\in \scr O(S)^\times$.
		\item $\eta[a]=[a]\eta$ for all $a\in \scr O(S)^\times$.
		\item $\eta h=0$, where $h=\eta[-1]+2$.
	\end{enumerate}
\end{theorem}

\begin{theorem}[The Grothendieck–Witt relations, {Corollary~\ref{cor:GW}}]
	The following relations hold in the group $\pi_0\Map(\1,\1)$:
	\begin{enumerate}
		\item $\langle ab^2\rangle = \langle a\rangle$ for all $a,b\in\scr O(S)^\times$.
		\item $\langle a\rangle+\langle -a\rangle=\langle 1\rangle + \langle -1\rangle$ for all $a\in\scr O(S)^\times$.
		\item $\langle a\rangle +\langle b\rangle = \langle a+b\rangle + \langle (a+b)ab\rangle$ for all $a,b\in\scr O(S)^\times$ with $a+b\in\scr O(S)^\times$.
	\end{enumerate}
\end{theorem}

After inverting $2$, the motivic sphere spectrum splits into the eigenspaces of the involution $\langle -1\rangle$:
\[
\1[\tfrac 12] = \1[\tfrac 12]^+ \times \1[\tfrac 12]^-.
\]
Over regular noetherian schemes, where K-theory is $\A^1$-invariant, Morel showed that the motivic spectrum $\L_{\A^1}(\1_{\Q}^+)$ represents Beilinson's rational motivic cohomology groups, defined as the eigenspaces of the Adams operations on the rational K-groups.
As proved in \cite[\sect 9]{atiyah}, the latter are represented by an idempotent motivic spectrum $\HH\Q$ in $\MS_S$ for any derived scheme $S$, and we obtain here the following generalization of Morel's result:

\begin{theorem}[Motivic Serre finiteness, {Corollary~\ref{cor:Q+}}]
	$\1_\Q^+=\HH\Q$.
\end{theorem}

The category of \emph{Beilinson motives} over $S$ is the full subcategory
\[
\Mod_{\HH\Q}(\MS_S)=(\MS_S)_\Q^+\subset \MS_S.
\] 
We conclude with several characterizations of Beilinson motives, generalizing the analogous results in $\A^1$-homotopy theory proved by Cisinski and Déglise in \cite[\sect 16]{CD}:

\begin{theorem}[Characterizations of Beilinson motives, {Theorem~\ref{thm:beilinson}}]
The following conditions are equivalent for a rational motivic spectrum $E$:
\begin{enumerate}
	\item $E$ is an $\HH\Q$-module.
	\item $E$ admits a structure of $\MGL$-module.
	\item $E$ is orientable.
	\item The swap automorphism of $\Sigma_{\P^1}^2 E$ is the identity.
	\item The involution $\langle -1\rangle\colon E\to E$ is the identity.
	\item The Hopf map $\eta\colon \Sigma_{\P^1} E\to \Sigma E$ is zero.
	\item $E$ satisfies étale descent.
\end{enumerate}
Moreover, these conditions always hold if 
none of the residue fields of $S$ are formally real and $-1$ is a unit sum of squares in $\scr O_{S,s}^{\rm h}$ whenever $\F_2$ is algebraically closed in $\kappa(s)$ (see Definition~\ref{def:unit-sum}).
\end{theorem}

\subsection*{Conventions}

We use throughout the Nisnevich-local version of $\MS$ as in \cite{atiyah}:
\[
\MS_S=\Sp_{\P^1}(\scr P_{\Nis,\sbu}(\Sm_S,\Sp)).
\]
However, with the exception of the statements about étale descent in Section~\ref{sec:beilinson}, all results in this paper hold in the simplest version of $\MS$ defined using Zariski descent and elementary blowup excision as in \cite{AHI} (provided one also removes the word ``henselian'' in Proposition~\ref{prop:-1}).

\subsection*{Acknowledgments}

I am grateful to Toni Annala and Ryomei Iwasa for many discussions around the topics of this article.
I also want to thank Alexey Ananyevskiy for 
some valuable correspondence about
Proposition~\ref{prop:a^2=1}, and Longke Tang for communicating his proof of Proposition~\ref{prop:normalization}.

\section{The multiplicative structure on the Tate circle}
\label{sec:Gm}

In stable $\A^1$-homotopy theory, the Tate circle $\Sigma^\infty_{\P^1}\G_\m$
inherits a multiplicative structure from the group scheme $\G_\m$.
More precisely, it is a nonunital $\E_\infty$-ring as the fiber of the augmention map $\Sigma^\infty_{\P^1}\G_{\m+}\to \1$. 
In this section, we show that the Tate circle $\Sigma^{-1}\P^1$ in $\MS$ splits off multiplicatively from $\G_\m$.

Let $\A^1_0=\P^1-\infty$ and $\A^1_\infty=\P^1-0$. Recall that $\Sigma^{-1}\P^1$ is a direct summand of $\G_\m$ in $\MS$, by the Bass fundamental theorem \cite[Proposition 4.12]{AHI}. In fact, the map
\[
\partial \colon \Sigma^{-1}\P^1\to \G_\m
\]
induced by the covering $\P^1=\A^1_0\cup \A^1_\infty$
(where everything is pointed at $1$)
admits a canonical retraction
\[
r\colon \G_\m\to \Sigma^{-1}\P^1,
\]
due to the fact that both embeddings $\A^1_0\into \P^1$ and $\A^1_\infty\into\P^1$ are canonically nullhomotopic \cite[Corollary 4.11]{AHI}.
The complementary summand to $\Sigma^{-1}\P^1$ in $\G_\m$ is $\A^1_0\oplus\A^1_\infty$.

\begin{remark}[Equivariant $\P$-homotopies]
	\label{rmk:equivariance}
	Let $M\in\Mon(\scr P(\Sm_S))$ be a presheaf of monoids acting on $X\in\scr P(\Sm_S)$, let $\scr E\in\Vect(X/M)$, and let $\sigma\colon \scr E\to\scr O_{X/M}$ be a linear map. Then the homotopy $h(\sigma)$ of \cite[Theorem 4.1(ii)]{AHI} between the composites
	\[
	\begin{tikzcd}[cramped, sep=scriptsize]
		X \ar[shift left]{r}{\sigma} \ar[shift right]{r}[swap]{0} & \V_X(\scr E) \ar[hook]{r} & \P_{X}(\scr E\oplus\scr O)
	\end{tikzcd}
	\]
	in $(\MS_S)_{/X_+}$ is $M$-equivariant, i.e., it lifts to a homotopy in $\Mod_{M_+}(\MS_S)_{/X_+}$.
	This follows from the fact that $h(\sigma)$ is $\scr P(\Sm_S)$-linear in $(X,\scr E,\sigma)$, which is a formal consequence of the functoriality of $h(\sigma)$ in the base $S$ and the smooth projection formula.
\end{remark}

\begin{construction}[The $\E_\infty$-decomposition of $\G_{\m+}$]
	\label{cst:Q}
	Consider the pullback square
	\[
	\begin{tikzcd}
		\G_{\m+} \ar{r} \ar{d} & \A^1_{0+} \ar{d} \\
		\A^1_{\infty+} \ar{r} & \P^1_+
	\end{tikzcd}
	\]
	in $\Mod_{\G_{\m+}}(\MS_S)_{/\1}$, and define $Q\in \Mod_{\G_{\m+}}(\MS_S)_{/\1}$ by the pullback square
	\[
	\begin{tikzcd}
		Q \ar{r} \ar{d} & \1 \ar{d}{0} \\
		\1 \ar{r}{\infty} & \P^1_+\rlap.
	\end{tikzcd}
	\]
	There is a canonical map $\partial\colon Q\to \G_{\m+}$ in $\Mod_{\G_{\m+}}(\MS_S)_{/\1}$.

	Consider the $\G_\m$-equivariant vector bundle $\pi_2\colon \A^1\times\A^1\to \A^1$, where $\G_\m$ acts diagonally on $\A^1\times\A^1$.
	By $\P$-homotopy invariance and Remark~\ref{rmk:equivariance}, the zero and diagonal sections $\A^1\rightrightarrows \P^1\times\A^1$
	become homotopic in $\Mod_{\G_{\m+}}(\MS_S)_{/\A^1_+}$.
	Composing with the projection $\P^1\times\A^1\to\P^1$, we obtain commuting triangles
	\[
	\begin{tikzcd}[column sep={1cm,between origins}]
		\A^1_{0+} \ar[hook]{rr} \ar{dr} & & \P^1_+ \\
		& \1 \ar{ur}[swap]{0} &
	\end{tikzcd}
	\qquad
	\begin{tikzcd}[column sep={1cm,between origins}]
		\A^1_{\infty+} \ar[hook]{rr} \ar{dr} & & \P^1_+ \\
		& \1 \ar{ur}[swap]{\infty} &
	\end{tikzcd}
	\]
	in $\Mod_{\G_{\m+}}(\MS_S)_{/\1}$. Together, these triangles define a retraction $r\colon \G_{\m+}\to Q$ of $\partial$ in $\Mod_{\G_{\m+}}(\MS_S)_{/\1}$.
	The composite $\partial\circ r\colon \G_{\m+}\to Q\to \G_{\m+}$ is thus an idempotent endomorphism in $\Mod_{\G_{\m+}}(\MS_S)_{/\1}$; it defines an idempotent element $e$ in the ring $\pi_0\Map(\1,\G_{\m+})$, whose image in $\pi_0\Map(\1,\1)$ is the identity.
	Hence, we obtain an $\E_\infty$-ring structure on $Q$ such that $r\colon \G_{\m+}\to Q$ is an $\E_\infty$-map, exhibiting $Q$ as the localization $\G_{\m+}[e^{-1}]$. Moreover, the augmentation $\G_{\m+}\to \1$ factors through $r$, since it sends $e$ to $1$.
	Together with the $\E_\infty$-map $\G_{\m+}\to \A^1_{0+}\times_\1 \A^1_{\infty+}$, we obtain an isomorphism of augmented $\E_\infty$-rings
\begin{equation}\label{eqn:G_m-decomposition}
\G_{\m+}= Q \times_\1\A^1_{0+}\times_\1 \A^1_{\infty+}.
\end{equation}
\end{construction}

\begin{construction}[The nonunital $\E_\infty$-ring structure on $\Sigma^{-1}\P^1$]
	\label{cst:nonunital}
	By definition of $Q$, there are canonical isomorphisms
	\[
	\fib(Q\to\1)=\Sigma^{-1}\fib(\P^1_+\to\1)=\Sigma^{-1}\P^1.
	\]
	As a result, the augmented $\E_\infty$-ring structure on $Q$ from Construction~\ref{cst:Q} induces a nonunital $\E_\infty$-ring structure on $\Sigma^{-1}\P^1$.
	The isomorphism~\eqref{eqn:G_m-decomposition} becomes a product decomposition
	\begin{equation}\label{eqn:G_m-nonunital}
	\G_\m= \Sigma^{-1}\P^1\times \A^1_0\times \A^1_\infty
	\end{equation}
	of nonunital $\E_\infty$-rings in $\MS_S$.
\end{construction}

\section{The Hopf map}

We define the \emph{Hopf map} $\eta\colon \1\to(\Sigma^{-1}\P^1)^{\otimes -1}$ so that 
\[
\eta\otimes\id_{(\Sigma^{-1}\P^1)^{\otimes 2}}\colon \Sigma^{-1}\P^1\otimes \Sigma^{-1}\P^1\to \Sigma^{-1}\P^1
\]
is the multiplication of the nonunital $\E_\infty$-ring structure from Construction \ref{cst:nonunital}.
In this section, we collect various facts about the Hopf map and prove in particular the ``easy'' Milnor–\allowbreak Witt relations in $\MS$.

Denote by $\epsilon\colon \1\to\1$ the Euler characteristic of the Tate circle $\Sigma^{-1}\P^1$, 
which is induced by 
the swap automorphism of $(\Sigma^{-1}\P^1)^{\otimes 2}$. 
	For a unit $a\in\scr O(S)^\times$, denote by $[a]\colon \1\to \Sigma^{-1}\P^1$ the composition
	\[
	\1\xrightarrow a \G_{\m+} \xrightarrow{r} \Sigma^{-1}\P^1,
	\]
	and by $\langle a\rangle\colon \1\to \1$ the multiplication by $a$ on $(\P^1,\infty)$.
	Note that $[1]=0$, $\langle 1\rangle=1$, and $\langle ab\rangle=\langle a\rangle\langle b\rangle$.

In the $\Z$-graded ring $\pi_0\Map(\1,(\Sigma^{-1}\P^1)^{\otimes *})$, if $\alpha$ has degree $m$ and $\beta$ has degree $n$, then
\[
\alpha\beta=\epsilon^{mn}\beta\alpha.
\]
In particular, $\epsilon$ and $\langle a\rangle$ are central. Since $\eta\otimes\id_{(\Sigma^{-1}\P^1)^{\otimes 2}}$ is a commutative pairing, we have
\begin{equation}
\label{eqn:epsilon}
\eta\epsilon = \epsilon\eta=\eta,
\end{equation}
which implies that $\eta$ is also central.

\begin{lemma}\label{lem:brackets}
	For any $a\in\scr O(S)^\times$, $\langle a\rangle =\eta[a]+1$.
\end{lemma}

\begin{proof}
	By~\eqref{eqn:G_m-nonunital}, the multiplication $m$ of $\G_{\m+}$ is zero on the summands $\A^1\otimes\Sigma^{-1}\P^1$.
	Thus, the restriction to $\Sigma^{-1}\P^1$ of the multiplication by $a$ on $\G_{\m+}$ is given by the composite
	\[
	\Sigma^{-1}\P^1 = \1\otimes\Sigma^{-1}\P^1 \xrightarrow{a\otimes \id} (\1\oplus \Sigma^{-1}\P^1)\otimes \Sigma^{-1}\P^1  \xrightarrow{m} \Sigma^{-1}\P^1.
	\]
	Since the last map $m$ is $\id_{\Sigma^{-1}\P^1}+\eta$, we obtain the desired formula.
\end{proof}

\begin{lemma}\label{lem:swap-GW}
	The swap automorphism of $\P^1\otimes\P^1$ is $\langle -1\rangle$. Equivalently, $\epsilon=-\langle -1\rangle$.
\end{lemma}

\begin{proof}
	 Recall that the Thom space construction $\scr E\mapsto\Th_S(\scr E)$ is functorial and symmetric monoidal in $\scr E$ \cite[\sect 3]{AHI}. Therefore, under the isomorphism
	 \[
	 \Th(\scr O)\otimes\Th(\scr O)\otimes\Th(\scr O) \simeq \Th(\scr O\oplus\scr O\oplus\scr O),
	 \]
	the swap automorphism and $\langle -1\rangle$ are identified with the automorphisms of $\Th(\scr O\oplus\scr O\oplus\scr O)$ given by the matrices
	\[
	\begin{pmatrix}
		0 & 1 & 0 \\
		1 & 0 & 0 \\
		0 & 0 & 1
	\end{pmatrix}
	\quad\text{and}\quad
	\begin{pmatrix}
		-1 & 0 & 0 \\
		0 & 1 & 0 \\
		0 & 0 & 1
	\end{pmatrix}.
	\] 
	Since $\SL_3(\Z)$ is perfect, it acts trivially on the $\otimes$-invertible object $\P^3/\P^2$ in $\h\MS_S$. In particular, the product of these two matrices acts as the identity. 
\end{proof}

\begin{theorem}[The Hopf relations]
	\label{thm:easy-MW}
	For any derived scheme $S$, the following relations hold in the graded ring $\pi_0\Map(\1_S,(\Sigma^{-1}\P^1)^{\otimes *})$:
	\begin{enumerate}
		\item $[ab]=[a]+[b]+\eta[a][b]$ for all $a,b\in \scr O(S)^\times$.
		\item $\eta[a]=[a]\eta$ for all $a\in \scr O(S)^\times$.
		\item $\eta h=0$, where $h=\eta[-1]+2$.
	\end{enumerate}
\end{theorem}

\begin{proof}
	(i) Since the multiplication $m$ of $\G_{\m+}$ is zero on the summands $\A^1\otimes\Sigma^{-1}\P^1$, by~\eqref{eqn:G_m-nonunital}, the map $[ab]$ decomposes as
	\[
	\1\otimes\1\xrightarrow{a\otimes b} (\1\oplus\Sigma^{-1}\P^1)\otimes (\1\oplus\Sigma^{-1}\P^1) \xrightarrow{m} \1\oplus \Sigma^{-1}\P^1 \to \Sigma^{-1}\P^1.
	\]
	The desired formula follows as $m=\id_\1+\id_{\Sigma^{-1}\P^1}+\id_{\Sigma^{-1}\P^1}+\eta$.
	
	(ii) As noted above, the element $\eta$ is central.
	
	(iii) By~\eqref{eqn:epsilon} and Lemma~\ref{lem:swap-GW}, we have $\eta=-\eta\langle -1\rangle$. By Lemma~\ref{lem:brackets}, we have $\langle -1\rangle=\eta[-1]+1$. Combining the two formulas gives the desired relation.
\end{proof}

\begin{corollary}\label{cor:GW2}
	For every $a\in\scr O(S)^\times$, we have
	$\langle a\rangle+\langle -a\rangle=\langle 1\rangle+\langle-1\rangle$ in $\pi_0\Map(\1_S,\1_S)$.
\end{corollary}

\begin{proof}
	By Lemma~\ref{lem:brackets}, this is equivalent to $\eta [a](\eta[-1]+2)=0$, which follows from Theorem \ref{thm:easy-MW}(ii,iii).
\end{proof}

Finally, we provide a more geometric description of the Hopf map (which is not used in the sequel):

\begin{proposition}\label{prop:hopf}
	Under the canonical isomorphism $\P^1\otimes\P^1\simeq \P^2/\P^1$, the Hopf map $\eta$ is the connecting map of the cofiber sequence
	$\P^1\to\P^2\to \P^2/\P^1$.
\end{proposition}

\begin{proof}
	Since the retraction $r\colon \G_{\m}\to \Sigma^{-1}\P^1$ is $\G_{\m}$-equivariant, $\eta$ is induced by the action of $\G_{\m}$ on $\P^1$:
	\[
	\eta\colon \Sigma^{-1}\P^1\otimes \P^1 \into \G_{\m+}\otimes\P^1\xrightarrow{\mathrm{act}}\P^1.
	\]
	Furthermore, by~\eqref{eqn:G_m-nonunital}, there is a commuting triangle
	\[
	\begin{tikzcd}
		\G_{\m+} \otimes \P^1 \ar{r}{\mathrm{act}} \ar[twoheadrightarrow]{d} & \P^1\rlap. \\
		(\1\oplus\Sigma^{-1}\P^1)\otimes\P^1 \ar[shorten >=5pt,xshift=5pt]{ur}[swap]{\id+\eta}
	\end{tikzcd}
	\]
	Hence, in the commutative diagram
	\[
	\begin{tikzcd}[column sep=1.5cm,/tikz/column 2/.style={column sep=1cm}]
		\G_{\m+}\otimes \P^1 \ar{r}{\delta\otimes\id} \ar[twoheadrightarrow]{d} & \G_{\m+}\otimes\G_{\m+}\otimes \P^1 \ar[twoheadrightarrow]{r} \ar[twoheadrightarrow]{d} & \G_{\m+}\otimes(\1\oplus\Sigma^{-1}\P^1)\otimes \P^1 \ar{r}{\id\otimes(\id+\eta)} \ar[twoheadrightarrow]{d} & \G_{\m+}\otimes \P^1 \ar[twoheadrightarrow]{d} \\
		\A^1_{0+}\otimes \P^1 \ar{r}{\delta\otimes\id} & \A^1_{0+}\otimes\A^1_{0+}\otimes \P^1 \ar[twoheadrightarrow]{r} & \A^1_{0+}\otimes \1\otimes \P^1 \ar{r}{\id\otimes \id} & \A^1_{0+}\otimes \P^1\rlap,
	\end{tikzcd}
	\]
	where $\delta$ is the diagonal, the top composite is the shear automorphism, while the bottom composite is the identity. This shows that the composite
	\[
	\A^1_{0+}\otimes\P^1\into \G_{\m+}\otimes \P^1 \xrightarrow{\mathrm{shear}} \G_{\m+}\otimes \P^1\onto \A^1_{0+}\otimes\P^1
	\]
	is the identity.
	
	By elementary blowup excision, the linear embedding $\P^1\to\P^2$ can be identified with the inclusion of the $0$-fiber $\P^1\to \Th_{\P^1}(\scr O(-1))$. 
	Consider the following diagram:
	\[
	\begin{tikzcd}[column sep=0cm,row sep=0.5cm]
		\Sigma^{-1}\Th_{\P^1}(\scr O(-1))/\P^1 \ar{rr} \ar{dr}[sloped]{\sim}[swap]{\alpha} \ar{dddd} & & (\A^1_0\oplus\Sigma^{-1}\P^1)\otimes\P^1 \ar{rr} \ar{dl}[yshift=1]{\mathrm{proj}} \ar{dd} & & 0 \ar{dd} \\ 
		 & \Sigma^{-1}\P^1\otimes\P^1  \\
		& & \G_{\m+}\otimes \P^1 \ar{rr}{\mathrm{inc}} \ar{dd}{\mathrm{shear}} & & \A^1_{\infty+}\otimes\P^1 \ar{dd} \\
		& \\
		\P^1 \ar{rr} \ar[equal]{dr} & & \A^1_{0+}\otimes\P^1 \ar{rr} \ar{dl} & & \Th_{\P^1}(\scr O(-1))\rlap. \\
		& \P^1 \ar[<-,crossing over]{uuuu}{\eta} \ar[<-,crossing over]{uuur}[near end]{\mathrm{act}}
	\end{tikzcd}
	\]
	The two squares on the right are cartesian, as is the back face of the prism.
	By the above observation about the shear map, the matrix of the middle column has the form
	\[
	\begin{pmatrix}
		\id & * \\ 0 & \eta
	\end{pmatrix}\colon(\A^1_0\otimes\P^1)\oplus (\Sigma^{-1}\P^1\otimes\P^1) \to (\A^1_{0}\otimes\P^1)\oplus \P^1,
	\]
	 so that the right face of the prism commutes and is cartesian. We therefore obtain the isomorphism $\alpha$ exhibiting the Hopf map as the fiber of $\P^1\to \Th_{\P^1}(\scr O(-1))$.
	
	To conclude, we must show that $\alpha$ is the canonical isomorphism. The latter is given by the zigzag of isomorphisms
	\[
	\Th_{\P^1}(\scr O(-1))/\P^1\leftarrow B/\partial B \rightarrow \P^1\otimes\P^1,
	\]
	where $B\to \P_{\P^1}(\scr O(-1)\oplus\scr O)$ is the blowup at $(0,0)$ and $B\to \P^1\times\P^1$ is the blowup at $(0,\infty)$.
	This zigzag extends to isomorphisms of cartesian squares
	\[
	\begin{tikzcd}[column sep=0.6cm]
		\G_{\m+}\otimes\P^1 \ar{r}{\mathrm{inc}} \ar{d}[swap]{\mathrm{shear}} & \A^1_{\infty+}\otimes\P^1 \ar{d} \ar[phantom]{dr}[description]{\Longleftarrow} & \G_{\m+}\otimes\P^1 \ar{r}{\mathrm{inc}} \ar{d} & \A^1_{\infty+}\otimes\P^1 \ar{d} \ar[phantom]{dr}[description]{\Longrightarrow} & \G_{\m+}\otimes\P^1 \ar{r}{\mathrm{inc}} \ar{d}[swap]{\mathrm{inc}} & \A^1_{\infty+}\otimes\P^1 \ar{d} \\
		\A^1_0\otimes\P^1 \ar{r} & \Th_{\P^1}(\scr O(-1))/\P^1 & B_0/\partial B_0 \ar{r} & B/\partial B & \A^1_0\otimes\P^1 \ar{r} & \P^1\otimes\P^1\rlap,
	\end{tikzcd}
	\]
	where the isomorphisms between the upper rows are identities. Passing to fibers, we deduce that the zigzag of isomorphisms between $\Sigma^{-1}\Th_{\P^1}(\scr O(-1))/\P^1$ and $\Sigma^{-1}\P^1\otimes\P^1$ identifies $\alpha$ with the identity, i.e., it coincides with $\alpha$.
\end{proof}

\section{The Steinberg relation}

The \emph{Steinberg embedding} is the pointed map
\[
\mathrm{st}\colon (\A^1-\{0,1\})_+\to \G_\m\times\G_\m,\quad a\mapsto (a,1-a).
\]
Here, $\G_\m$ is pointed at $1$, so that $\mathrm{st}$ sends the base point to $(1,1)$. 

\begin{theorem}[The Steinberg relation]
	\label{thm:steinberg}
	For any derived scheme $S$, the composite
	\[
	(\A^1-\{0,1\})_+\xrightarrow{\mathrm{st}} \G_\m\wedge\G_\m \xrightarrow{r\wedge r} \Sigma^{-2}(\P^1\wedge\P^1)
	\]
	is nullhomotopic in $\MS_S$. 
	Hence, for any $a\in\scr O(S)^\times$ such that $1-a\in\scr O(S)^\times$, we have 
	$[a][1-a]=0$
	in $\pi_0\Map(\1_S,(\Sigma^{-1}\P^1)^{\otimes 2})$.
\end{theorem}

\begin{corollary}\label{cor:steinberg}
	Let $a,b\in\scr O(S)^\times$ be units such that $a+b\in \scr O(S)^\times$. Then we have
	\[
	\langle a\rangle +\langle b\rangle = \langle a+b\rangle + \langle \frac{ab}{a+b}\rangle
	\]
	in $\pi_0\Map(\1_S,\1_S)$.
\end{corollary}

\begin{proof}
	Dividing through by $a+b$, we can assume $b=1-a$. We must then prove
	\[
	\langle a\rangle+\langle 1-a\rangle=1+\langle a(1-a)\rangle.
	\]
	This follows from Theorem~\ref{thm:easy-MW}(i), Theorem~\ref{thm:steinberg}, and Lemma~\ref{lem:brackets}.
\end{proof}

The rest of this section is devoted to the proof of Theorem~\ref{thm:steinberg}.
Recall that we denote by $\A^1_0=\P^1-\infty$ and $\A^1_\infty=\P^1-0$ the standard affine subschemes of $\P^1$, which are always pointed at $1$ in what follows.
We will work in the ambient space
\[
P =  (\P^1\times\P^1)-\{(\infty,0),(0,\infty)\} = (\A^1_0\times\A^1_0)\cup (\A^1_\infty\times\A^1_\infty) \subset \P^1\times\P^1
\]
with the four affine axes
\begin{alignat*}{2}
	X_0& =\A^1_0\times 0, & Y_0&=0\times\A^1_0,\\
	X_\infty&= \A^1_\infty\times\infty, &\qquad Y_\infty&=\infty\times\A^1_\infty.
\end{alignat*}
We define the following points and lines in $P$:
\begin{alignat*}{3}
	x&=(1,0), &\quad y&=(0,1), &\quad w&=(1,1), \\
	u&=(1,\infty), & v&=(\infty,1), & z&=(\infty,\infty), \\
	U&= xwu, & V&=ywv, & Z&=xyz.
\end{alignat*}
We further define $T$ to be the triangle formed by the three lines $U,V,Z$ and pointed at $w$:
\[
T=\colim(x\sqcup w\sqcup y\rightrightarrows U \sqcup Z\sqcup V).
\]
The following figures depict the situation near $(0,0)$ and near $(\infty,\infty)$:
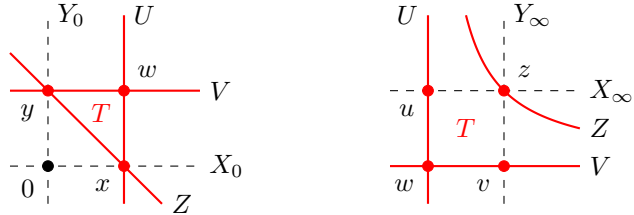
\[
\begin{tikzpicture}
	\draw[draw=red,thick] (0,2cm) -- (2cm,0) node[right]{$Z$};
	\draw[draw=red,thick] (0,1.5cm) -- (2.5cm,1.5cm) node[right]{$V$};
	\draw[draw=red,thick] (1.5cm,0) -- (1.5cm,2.5cm) node[right]{$U$};
	\draw[dashed] (.5cm,0) -- (.5cm, 2.5cm) node[right]{$Y_0$};
	\draw[dashed] (0,.5cm) -- (2.5cm, .5cm) node[right]{$X_0$};
	\node[circle, draw, fill, inner sep=0, minimum size=4pt, label=below left:{$0$}] at (0.5cm,0.5cm) {};
	\node[circle, draw=red, fill=red, inner sep=0, minimum size=4pt, label=above right:{$w$}] at (1.5cm,1.5cm) {};
	\node[circle, draw=red, fill=red, inner sep=0, minimum size=4pt, label=below left:{$y$}] at (0.5cm,1.5cm) {};
	\node[circle, draw=red, fill=red, inner sep=0, minimum size=4pt, label=below left:{$x$}] at (1.5cm,0.5cm) {};
	\node[red] at (1.2cm,1.2cm) {$T$};
	\begin{scope}[xshift=5cm]
	\draw[draw=red,thick] plot [
	        samples=100,
	        domain=1:2.5
	        ] (\x,{0.5+1/(\x-0.5)}) node[right]{$Z$};
	\draw[draw=red,thick] (0,0.5cm) -- (2.5cm,0.5cm) node[right]{$V$};
	\draw[draw=red,thick] (0.5cm,0) -- (0.5cm,2.5cm) node[left]{$U$};
	\draw[dashed] (1.5cm,0) -- (1.5cm, 2.5cm) node[right]{$Y_\infty$};
	\draw[dashed] (0,1.5cm) -- (2.5cm, 1.5cm) node[right]{$X_\infty$};
	\node[circle, draw=red, fill=red, inner sep=0, minimum size=4pt, label=above right:{$z$}] at (1.5cm,1.5cm) {};
	\node[circle, draw=red, fill=red, inner sep=0, minimum size=4pt, label=below left:{$w$}] at (.5cm,.5cm) {};
	\node[circle, draw=red, fill=red, inner sep=0, minimum size=4pt, label=below left:{$v$}] at (1.5cm,0.5cm) {};
	\node[circle, draw=red, fill=red, inner sep=0, minimum size=4pt, label=below left:{$u$}] at (0.5cm,1.5cm) {};
	\node[red] at (1cm,1cm) {$T$};
	\end{scope}
\end{tikzpicture}
\]
Finally, we set
\begin{align*}
& T_0 = T\cap(\A^1_0\times\A^1_0) = T-\{u,v,z\}, \\
& \bar P = P/(U\vee V).
\end{align*}

Consider the following diagram of cofiber sequences in $\MS_S$:
\[
\begin{tikzcd}[column sep=0.5cm]
	(Z-\{x,y,z\})_+ \ar{r} \ar[equal]{d} & 0 \ar{d} \ar{r} & \Sigma((Z-\{x,y,z\})_+) \ar{d} \\
	(Z-\{x,y,z\})_+ \ar{r} \ar{d}[swap]{\mathrm{st}} & T \ar{d} \ar{r} & T/((Z-\{x,y,z\})_+) \ar{d} \\
	\G_\m\wedge \G_\m \ar{r} \ar{d}[swap]{r \wedge r} & \bar P \ar{r} \ar{d} & \bar P/(\G_\m\wedge\G_\m) \ar{d} \\
	\Sigma^{-2}(\P^1\wedge\P^1) \ar{r} & 0 \ar{r} & \Sigma^{-1}(\P^1\wedge\P^1)\rlap.
\end{tikzcd}
\]
The top left square commutes using the $\P$-homotopy between the embedding $Z-z\into Z$ and the constant map at $y$ \cite[Corollary 4.11]{AHI} and the $\P^1$-homotopy between $y$ and $w$ in $V$.
The composition of the third column is thus the suspension of the Steinberg map.

\begin{remark}
	We could alternatively make the top left square commute using $x$ and $U$ instead of $y$ and $V$. However, it is crucial that we contract the complement of the point $z$ at infinity: the following proof does not work if we instead contract $Z-x\into Z$ to $y$.
\end{remark}

We have Zariski pushout squares
\begin{alignat*}{2}
T &= T_0 \cup (T-\{x,y\})\quad& \text{with intersection}\quad& T_0-\{x,y\},\\
\bar P &= (\A^1_0\wedge\A^1_0) \cup (\A^1_\infty\wedge \A^1_\infty)\quad& \text{with intersection}\quad& \G_\m\wedge\G_\m,
\end{alignat*}
compatible with the map $T\to\bar P$, inducing coproduct decompositions
\begin{align*}
	T/(T_0-\{x,y\}) &= T_0/(T_0-\{x,y\}) \oplus (T-\{x,y\})/(T_0-\{x,y\}),\\
	\bar P/(\G_\m\wedge\G_\m) & = (\A^1_0\wedge\A^1_0)/(\G_\m\wedge\G_\m)\oplus (\A^1_\infty\wedge \A^1_\infty)/(\G_\m\wedge\G_\m).
\end{align*}
The suspension of the Steinberg map is therefore the sum of the two composites
\begin{equation}\label{eqn:susp-Steinberg}
\begin{tikzcd}[column sep=-1cm]
	& \Sigma((Z-\{x,y,z\})_+) \ar{dl}\ar{dr} & \\
	T_0/(T_0-\{x,y\}) \ar{d} & & (T-\{x,y\})/(T_0-\{x,y\}) \ar{d} \\
	(\A^1_0\wedge\A^1_0)/(\G_\m\wedge\G_\m) \ar{dr} & & (\A^1_\infty\wedge \A^1_\infty)/(\G_\m\wedge\G_\m) \ar{dl}  \\
	& \Sigma^{-1}(\P^1\wedge\P^1)\rlap. &
\end{tikzcd}
\end{equation}
We will show that both composites are nullhomotopic using Tang's Gysin maps \cite{Tang} (see also \cite[Theorem 2.3]{atiyah} for a summary of their main properties).

\emph{The left composite in~\eqref{eqn:susp-Steinberg} is nullhomotopic.}
We show that the map $T_0/(T_0-\{x,y\})\to \Sigma^{-1}(\P^1\wedge\P^1)$ is nullhomotopic.
Consider the square
\[
\begin{tikzcd}
	T_0/(T_0-\{x,y\}) \ar{d} \ar{r}{\gys} & \Th_x(\scr N_{x/U})\sqcup_{x_+}\Th_{x\sqcup y}(\scr N_{x\sqcup y/Z})\sqcup_{y_+} \Th_y(\scr N_{y/V}) \ar{d}{\nu} \\
	(\A^1_0\wedge\A^1_0)/(\G_\m\wedge\G_\m) \ar{r}{\gys} & \Th_{X_0\vee Y_0}(\scr N_{X_0\vee Y_0/P})/\Th_{x\sqcup y}(\scr N_{x\sqcup y/U\vee V}),
\end{tikzcd}
\]
which commutes by the functoriality of Gysin maps.
Here, $\Th_{X_0\vee Y_0}(\scr N_{X_0\vee Y_0/P})$ is defined by the pullback square
\[
\begin{tikzcd}
	\Th_{X_0\vee Y_0}(\scr N_{X_0\vee Y_0/P}) \ar{r} \ar{d} & \Th_{Y_0}(\scr N_{Y_0/P}) \ar{d}{\gys} \\
	\Th_{X_0}(\scr N_{X_0/P}) \ar{r}{\gys} & \Th_0(\scr N_{0/P})\rlap,
\end{tikzcd}
\]
and the Gysin map from $\A^1_0\wedge\A^1_0$ is defined using the functoriality of Gysin maps in triples \cite[Theorem 2.3(ii)]{atiyah}.
By Lemma~\ref{lem:gys-factorization} below, the map $(\A^1_0\wedge\A^1_0)/(\G_\m\wedge\G_\m)\to \Sigma^{-1}(\P^1\wedge\P^1)$ factors through the Gysin map. Since $\nu$ is trivially nullhomotopic, this proves the claim.

\emph{The right composite in~\eqref{eqn:susp-Steinberg} is nullhomotopic.}
By construction, the map 
\[
\Sigma((Z-\{x,y,z\})_+)\to (T-\{x,y\})/(T_0-\{x,y\})\simto T/T_0
\] 
factors through $\Sigma((Z-z)_+)$, and we show that the map $\Sigma((Z-z)_+)\to \Sigma^{-1}(\P^1\wedge\P^1)$ is nullhomotopic.
There is a Zariski-local coproduct decomposition
\[
T/T_0 = Z/(Z-z) \oplus U/(U-u) \oplus V/(V-v).
\]
We consider as before the commuting square
\[
\begin{tikzcd}
	Z/(Z-z) \oplus U/(U-u) \oplus V/(V-v) \ar{d} \ar{r}{\gys} & \Th_z(\scr N_{z/ Z}) \oplus \Th_u(\scr N_{u/ U})\oplus\Th_v(\scr N_{v/ V}) \ar{d}{\mu} \\
	(\A^1_\infty\wedge \A^1_\infty)/(\G_\m\wedge\G_\m) \ar{r}{\gys} & \Th_{X_\infty\vee Y_\infty}(\scr N_{X_\infty\vee Y_\infty/P})/\Th_{u\sqcup v}(\scr N_{u\sqcup v/ U\vee V})\rlap.
\end{tikzcd}
\]
By Lemma~\ref{lem:gys-factorization}, the map $(\A^1_\infty\wedge \A^1_\infty)/(\G_\m\wedge\G_\m)\to \Sigma^{-1}(\P^1\wedge\P^1)$ factors through the Gysin map. 
The map $\mu$ is clearly nullhomotopic on the last two summands, but not on the first one. 
Instead, we claim that the composite
\begin{equation*}\label{eqn:gysin-triv}
\Sigma((Z-z)_+)\to Z/(Z-z) \xrightarrow{\gys} \Th_z(\scr N_{z/ Z})
\end{equation*}
is nullhomotopic, which will complete the proof. 
Here, the first map is induced by the nullhomotopy of the map $(Z-z)_+\to Z-z\into Z$, where $Z$ is pointed at $y$ (by the original choice of nullhomotopy in $T$). If we identify $Z$ with $\P^1$ so that $x,y,z$ become $0,1,\infty$, then, by Lemma~\ref{lem:gysin-A^1}, the above composite can be identified with
\[
\Sigma(\A^1_{0+}) \to \Sigma\A^1_0\to  \P^1/\A^1_0 \simeq \A^1_\infty/\G_\m\to \Sigma\G_\m\xrightarrow{r} \P^1,
\]
which is nullhomotopic as $\Sigma\A^1_0\to \A^1_\infty/\G_\m \to \P^1$ is even a split cofiber sequence.

\begin{lemma}\label{lem:gysin-A^1}
	Under the isomorphism $\P^1=(\P^1,1)\simeq (\P^1,\infty)=\Th(\scr O)$ given by $\left(\begin{smallmatrix}1 & 0 \\ -1 & 1\end{smallmatrix}\right)$, the Gysin map
	\[
	\A^1/\G_\m\xrightarrow{\gys_0} \Th_0(\scr N_{0/\A^1})
	\]
	coincides with the composition
	\[
	\A^1/\G_\m\to \Sigma\G_\m\xrightarrow{r}\P^1.
	\]
\end{lemma}

\begin{proof}
	This is the special case $\scr E=\scr O$ of Proposition~\ref{prop:normalization}.
\end{proof}

\begin{lemma}\label{lem:gys-factorization}
	The map
	\[
	(\A^1\wedge\A^1)/(\G_\m\wedge\G_\m) \to\Sigma(\G_\m\wedge\G_\m) \xrightarrow{r\wedge r}\Sigma^{-1}(\P^1\wedge\P^1)
	\]
	factors through the Gysin map
	\[
	\gys_{X\vee Y/\A^2}\colon (\A^1\wedge\A^1)/(\G_\m\wedge\G_\m) \to \Th_{X\vee Y/x\sqcup y}(\scr N_{X\vee Y/\A^2}).
	\]
\end{lemma}

\begin{proof}
	Recalling the definition of $\Th_{X\vee Y}$ as a pullback, we consider the cartesian squares
	\[
	\begin{tikzcd}[column sep=0.5cm]
		(\A^1\wedge\A^1)/(\G_\m\wedge\G_\m) \ar{r} \ar{d} & (\A^1\wedge\A^1)/(\G_\m\wedge\A^1) \ar{d} \ar[phantom]{dr}[description]{\overset{\gys}\Longrightarrow} & \Th_{X\vee Y/x\sqcup y}(\scr N_{X\vee Y/\A^2}) \ar{r} \ar{d} & \Th_{Y/y}(\scr N_{Y/\A^2}) \ar{d}{\gys} \\
		(\A^1\wedge\A^1)/(\A^1\wedge\G_\m) \ar{r} \ar[phantom]{dr}[description]{\Downarrow} & (\A^1\wedge\A^1)/(\A^1\wedge\A^1-0) & \Th_{X/x}(\scr N_{X/\A^2}) \ar{r}{\gys} \ar[phantom]{dr}[description]{\Downarrow ?} & \Th_0(\scr N_{0/\A^2}) \\
		\Sigma(\G_\m\wedge\G_\m) \ar{r} \ar{d} & \Sigma(\G_\m\wedge\A^1) \ar{d} \ar[phantom]{dr}[description]{\Longrightarrow} & \Sigma^{-1}(\P^1\wedge\P^1) \ar{r} \ar{d} & 0 \ar{d} \\
		\Sigma(\A^1\wedge \G_\m) \ar{r} & \Sigma(\A^1\wedge\A^1-0) & 0 \ar{r} & \P^1\wedge\P^1\rlap.
	\end{tikzcd}
	\]
	To get the desired factorization, we see that it suffices to complete the following cube:
	\[
	\begin{tikzcd}[column sep={0.5cm},row sep={1cm,between origins},/tikz/column 1/.style={column sep=-0.5cm}]
		(\A^1\wedge\A^1)/(\G_\m\wedge\A^1) \ar{rr}{\gys} \ar{dr} \ar{dd} & & \Th_{Y/y}(\scr N_{Y/\A^2}) \ar{dd}[pos=0.2]{\gys} \ar[dashed]{dr} \\
		& \Sigma(\G_\m\wedge\A^1) \ar[crossing over,shorten >=5pt]{rr} & & 0  \ar{dd} \\
		(\A^1\wedge\A^1)/(\A^1\wedge\A^1-0) \ar{rr}[pos=0.6]{\gys} \ar{dr} & & \Th_0(\scr N_{0/\A^2}) \ar[dashed]{dr} \\
		& \Sigma(\A^1\wedge\A^1-0) \ar{r} \ar[<-,crossing over]{uu} & \Sigma^2(\G_\m\wedge\G_\m) \ar{r} & \P^1\wedge\P^1\rlap.
	\end{tikzcd}
	\]
	The front face is by definition the nullhomotopy of $r\wedge i \colon \Sigma\G_\m\wedge \A^1\to\P^1\wedge\P^1$ induced by the null\-ho\-mo\-topy of the inclusion $i \colon \A^1\into \P^1$. Identifying $\Th(\scr O)$ with $(\P^1,1)$ and using the linearity of Gysin maps, we can rewrite the back face as
	\[
	\begin{tikzcd}[column sep=2cm]
		\A^1/\G_\m\wedge \A^1 \ar{r}{\gys_0\wedge \id} \ar{d} & \P^1\wedge \A^1 \ar{d}{\id\wedge \gys_0} \\
		\A^1/\G_\m\wedge \A^1/\G_\m \ar{r}{\gys_0\wedge\gys_0} & \P^1\wedge \P^1\rlap.
	\end{tikzcd}
	\]
	By Lemma~\ref{lem:gysin-A^1}, we can thus take the lower dashed arrow to be the identity to make the bottom face commute.
	To complete the cube, we must show that the given nullhomotopy of the composite 
	\[
	\A^1/\G_\m\wedge \A^1\to\Sigma\G_\m\wedge \A^1\xrightarrow{r\wedge i } \P^1\wedge\P^1
	\] is induced by some nullhomotopy of the vertical Gysin map $\id\wedge\gys_0\colon \P^1\wedge \A^1\to \P^1\wedge\P^1$.
	Again by Lemma~\ref{lem:gysin-A^1}, the latter coincides with $\id\wedge i $, so this is true by definition.
\end{proof}

A key ingredient in the above argument is the following refinement of the normalization property of Gysin maps \cite[Theorem 2.3(iii)]{AHI}, whose proof we learned from Longke Tang:

\begin{proposition}[Tang]
	\label{prop:normalization}
	Let $\scr E$ be a finite locally free sheaf on $X\in\Sm_S$. Then the Gysin map
	\[
	\frac{\P(\scr E\oplus\scr O)}{\P(\scr E\oplus \scr O)-0} = \frac{\V(\scr E)}{\V(\scr E)-0} \xrightarrow{\gys} \Th_0(\scr E) = \frac{\P(\scr E\oplus\scr O)}{\P(\scr E)}
	\]
	of the zero section $X\into \V(\scr E)$
	coincides, naturally in $(S,X,\scr E)$, with the map induced by the commuting triangle
	\[
	\begin{tikzcd}
		\P(\scr E \oplus \scr O)-0 \ar[hook]{dr} \ar[twoheadrightarrow]{d} & \\
		\P(\scr E) \ar[hook]{r} & \P(\scr E\oplus\scr O)
	\end{tikzcd}
	\]
	of \cite[Proposition 4.9]{AHI}.
\end{proposition}

\begin{proof}
	Recall that the Gysin map $\gys_{X/Y}\colon Y_+\to\Th_X(\scr N_{X/Y})$ of a closed immersion $X\into Y$ in $\Sm_S$ is constructed using the cofiber sequence
	\[
	\Th_X(\scr N_{X/Y}) \to \frac{\Bl_{X\times 0}(Y\times\P^1)}{\Bl_{X\times 0}(Y\times 0)} \to \frac{Y\times\P^1}{Y\times 0}
	\]
	in $\scr P_{\sbu}(\Sm_S)_*$: it is induced by the embedding $i_1\colon Y\into \Bl_{X\times 0}(Y\times\P^1)$ and the $\P^1$-homotopy between $i_1$ and $i_0$ in $Y\times\P^1$.
	Moreover, the restriction to $Y-X$ of this $\P^1$-homotopy lifts to the blowup, which defines a nullhomotopy of the restriction to $Y-X$ of $\gys_{X/Y}$, hence
	\[
	\gys_{X/Y}\colon Y/(Y-X) \to \Th_X(\scr N_{X/Y}).
	\]
	
	When $Y=\V(\scr E)$, this cofiber sequence is split by the retraction
	\[
	q\colon \Bl_{X\times 0}(\V(\scr E)\times\P^1) \to \P(\scr E\oplus \scr O),\quad (v,t)\mapsto [v:t].
	\]
	More precisely, if $\P^1=\P(\scr O_0\oplus\scr O_\infty)$, a point of the blowup consists of $\sigma\colon\scr E\to\scr O$, $(x,y)\colon\scr O_0\oplus\scr O_\infty\onto\scr L$, and $\phi\colon (\scr E\otimes\scr L)\oplus\scr O_0\onto\scr M$ such that $(\sigma,x)$ vanishes on $\ker\phi$, and its image by $q$ is the surjection 
	\[
	\scr E\oplus\scr O=(\scr E\otimes\scr O_\infty)\oplus\scr O_0 \xrightarrow{(y,\id)}(\scr E\otimes\scr L)\oplus\scr O_0 \xrightarrow{\phi} \scr M.
	\]
	This canonically identifies the Gysin map $\gys\colon \V(\scr E)_+\to \Th_X(\scr E)$ with the composite
	\[
	\V(\scr E)_+\xrightarrow{i_1} \frac{\Bl_{X\times 0}(\V(\scr E)\times\P^1)}{\Bl_{X\times 0}(\V(\scr E)\times 0)}\xrightarrow{q}\Th_X(\scr E),
	\]
	which is the quotient map $\V(\scr E)_+\into \P(\scr E\oplus\scr O)_+\to\Th_X(\scr E)$.
	The nullhomotopy of the restriction of $\gys$ to $\V(\scr E)-0$ is then defined by the image under $q$ of the tautological $\P^1$-homotopy between 
	\[
	i_1,i_0\colon (\V(\scr E)-0)_+ \to \Bl_{X\times 0}(\V(\scr E)\times\P^1),
	\]
	which is the $\P^1$-homotopy
	\[
	(\V(\scr E)-0)\times \P^1 \to \P(\scr E\oplus\scr O),\quad (v,t)\mapsto [v:t],
	\]
	between the inclusion $\V(\scr E)-0\into \P(\scr E\oplus\scr O)$ and the composite $\V(\scr E)-0\onto \P(\scr E)\into \P(\scr E\oplus\scr O)$.
	Thus, it remains to observe that this $\P^1$-homotopy is precisely the restriction to $\V(\scr E)-0$ of the twisted $\P^1$-homotopy
	\[
	\P_{\P(\scr E\oplus\scr O)-0}(\scr O(-1)\oplus\scr O)\to \P(\scr E\oplus \scr O)
	\]
	from the proof of \cite[Proposition 4.9]{AHI}.
\end{proof}

\begin{remark}
	This proof of the Steinberg relation simplifies significantly in $\MS_S^{\A^1}$: first, we can replace $T$ and $\bar P$ by $T_0$ and $\A^1_0\wedge \A^1_0$, and second,
	the Gysin maps become purity isomorphisms,
	so the fact that $\nu$ is nullhomotopic directly gives the result. 
	It shows in fact that the suspended Steinberg map
	\[
	\Sigma((\A^1-\{0,1\})_+)\to \Sigma(\G_\m\wedge\G_\m)
	\]
	is nullhomotopic in $\scr P_{\Zar,\A^1}(\Sm_S)$.
	
	The idea to use the purity isomorphism also appears in an unpublished proof of the Steinberg relation by Powell \cite{Powell}, which however like the proof of Hu and Kriz \cite{HuKriz} is flawed: in both cases it is only argued that the map $\Sigma(\A^1-\{0,1\})\to \Sigma(\G_\m\wedge\G_\m)$ is nullhomotopic in $\scr P_{\Zar,\A^1}(\Sm_S)$, which does not imply the Steinberg relation.
	Valid proofs in $\A^1$-homotopy theory based on the idea of Hu and Kriz were given in \cite[\sect 2]{DruzhininMW} and \cite{HoyoisSteinberg}.
\end{remark}

\section{The Grothendieck–Witt relations}

Let $\rm E_2(S)$ be the subgroup of $\SL_2(S)$ generated by elementary matrices, and let $\rm E_2'(S)$ be its normal closure in $\GL_2(S)$ (which is strictly larger in general).
We observe that the matrix $\left(\begin{smallmatrix}a & 0 \\ 0 & a^{-1}\end{smallmatrix}\right)$ belongs to $\rm E_2'(S)$ for any unit $a\in\scr O(S)^\times$. Indeed, 
\[
\begin{pmatrix}
	a & 0 \\ 0 & a^{-1}
\end{pmatrix}
=
\left[\begin{pmatrix}
	a & 0 \\ 0 & 1
\end{pmatrix},
\begin{pmatrix}
	0 & 1 \\ -1 & 0
\end{pmatrix}\right]
\quad\text{and}\quad
\begin{pmatrix}
	0 & 1 \\ -1 & 0
\end{pmatrix}
=
\begin{pmatrix}
	1 & 1 \\ 0 & 1
\end{pmatrix}
\begin{pmatrix}
	1 & 0 \\ -1 & 1
\end{pmatrix}
\begin{pmatrix}
	1 & 1 \\ 0 & 1
\end{pmatrix}
\in \rm E_2(S).
\]

\begin{lemma}\label{lem:E2}
	The group $\rm E_2'(S)$ acts trivially on $\P^1_+$ in $\h\MS_S$.
\end{lemma}

\begin{proof}
	It suffices to show that the matrix $\left(\begin{smallmatrix}1 & f \\ 0 & 1\end{smallmatrix}\right)$ acts trivially for all $f\in \scr O(S)$.
	This is explicitly done in the proof of $\P$-homotopy invariance \cite[Theorem 4.1(ii)]{AHI}.
\end{proof}

The following result is a consequence of the already established Milnor–Witt relations when $\pi_0\scr O(S)$ is a field or a local ring whose residue field has at least $4$ elements (see Remark~\ref{rmk:KMW}(ii)), but not in general: 

\begin{proposition}
	\label{prop:a^2=1}
	For any $a\in\scr O(S)^\times$, we have $\langle a^2\rangle=1$ in $\MS_S$.
\end{proposition}

\begin{proof}
	The endomorphism $\langle a^2\rangle$ of $(\P^1,\infty)$ is a retract of the endomorphism of $\P^1_+$ induced by the matrix $\left(\begin{smallmatrix}a & 0 \\ 0 & a^{-1}\end{smallmatrix}\right)$, which lies in $\rm E_2'(S)$, so the result follows from Lemma~\ref{lem:E2}.
\end{proof}

\begin{corollary}[The Grothendieck–Witt relations]
	\label{cor:GW}
	For any derived scheme $S$, the following relations hold in the group $\pi_0\Map(\1_S,\1_S)$:
	\begin{enumerate}
		\item $\langle ab^2\rangle = \langle a\rangle$ for all $a,b\in\scr O(S)^\times$.
		\item $\langle a\rangle+\langle -a\rangle=\langle 1\rangle + \langle -1\rangle$ for all $a\in\scr O(S)^\times$.
		\item $\langle a\rangle +\langle b\rangle = \langle a+b\rangle + \langle (a+b)ab\rangle$ for all $a,b\in\scr O(S)^\times$ with $a+b\in\scr O(S)^\times$.
	\end{enumerate}
\end{corollary}

\begin{proof}
	(i) follows from Proposition~\ref{prop:a^2=1}, (ii) was already proved in Corollary~\ref{cor:GW2}, and (iii) follows from Corollary~\ref{cor:steinberg} and (i).
\end{proof}

\begin{definition}\label{def:unit-sum}
	Let $R$ be a commutative ring and let $n\geq 0$.
	An element $a\in R^\times$ is called a \emph{unit sum of squares} of exponent $\leq n$ if it is a square or if $n\geq 1$ and $a=b+c$, where $b,c\in R^\times$ are unit sums of squares of exponent $\leq n-1$.
\end{definition}

\begin{lemma}\label{lem:unit-sum}
	If $a\in\scr O(S)^\times$ is a unit sum of squares of exponent $\leq n$, then $2^n(\langle a\rangle-1)=0$ in $\MS_S$.
\end{lemma}

\begin{proof}
	This follows from Corollary~\ref{cor:GW}(i,iii) by induction on $n$.
\end{proof}

\begin{proposition}\label{prop:-1}
	If $-1$ is a unit sum of squares in each henselian local ring of $S$, then $\langle -1\rangle=1$ in $\MS_S[\tfrac 12]$.
\end{proposition}

\begin{proof}
	Since $\langle -1\rangle^2=1$, the equality $\langle -1\rangle=1$ holds in $\MS_S[\tfrac 12]$ if and only if $1+\langle -1\rangle$ is a unit, which is a Nisnevich-local property.
	By continuity, we may thus assume that $-1$ is a unit sum of squares in $\scr O(S)$.
	Then the result follows from Lemma~\ref{lem:unit-sum}.
\end{proof}

\begin{remark}\label{rmk:unit-sum}
	Let $R$ be a henselian local ring with residue field $k$. 
	Then $-1$ is a unit sum of squares in $R$ in the following cases:
	\begin{enumerate}
		\item $k$ has characteristic zero and is not formally real.
		\item $k$ has odd characteristic.
		\item $k$ contains $\F_{2^n}$ for some $n\geq 2$. 
	\end{enumerate}
	In cases (i) and (ii), $-1$ is a sum of squares in $k$ and we can lift this sum to $R$, while in case (iii), $-1$ is a unit sum of $(2^n-1)$st roots of unity in $R$.
	If $\F_2$ is algebraically closed in $k$,
	the property may or may not hold: 
	it holds in $\Z[i]_{(1+i)}^{\rm h}$ but fails in $\Z/4$. 
	We do not know whether $\langle -1\rangle =1$ in $\MS_{\Z/4}[\tfrac 12]$.
\end{remark}

\begin{remark}[Milnor–Witt K-theory]
	\label{rmk:KMW}
	Let $S$ be a derived scheme.
	\begin{enumerate}
		\item Let $\K^\MW_*(S)$ be the graded ring generated by $\eta$ in degree $-1$ and $\{[a]\suchthat a\in \pi_0\scr O(S)^\times\}$ in degree $1$, modulo the Hopf relations of Theorem~\ref{thm:easy-MW} and the Steinberg relations of Theorem~\ref{thm:steinberg}. By these theorems, there is a map of graded rings
	\[
	\K^\MW_*(S) \to \pi_0\Map(\1_S,(\Sigma^{-1}\P^1)^{\otimes *}).
	\]
	Composing with $\A^1$-localization $\MS_S\to \MS_S^{\A^1}$, we recover the map constructed by Druzhinin in \cite[Theorem 1.6]{DruzhininMW}.
	Furthermore, if $S$ is a field or a regular local ring over an infinite field of characteristic not $2$, $\A^1$-localization provides a retraction of this map, by Morel's theorem \cite[Corollary 6.3]{Morel} and \cite[Theorem 6.3]{GilleScullyZhong}.
	\item In \cite[Lemmas 4.14 and 4.16]{SchlichtingEuler}, Schlichting gives explicit presentations of the ring $\K^\MW_0(S)$ and of the module $\K^\MW_1(S)$. We have in particular $\K^\MW_0(S)=\Z[\pi_0\scr O(S)^\times]/I$, where $\pi_0\scr O(S)^\times\to \K^\MW_0(S)$ sends $a$ to $\langle a\rangle =\eta[a]+1$ and $I$ is the ideal generated by the relations (ii) and (iii) of Corollary~\ref{cor:GW}. 
	Let $\tK^\MW_*(S)$ be the quotient of $\K^\MW_*(S)$ by the further relations $\eta[a^2]=0$ for all $a\in\scr O(S)^\times$, so that $\tK^\MW_0(S)=\Z[\pi_0\scr O(S)^\times]/\bar I$ where $\bar I$ is generated by the relations (i)–(iii) of Corollary~\ref{cor:GW}.
	We thus get a map of graded rings
	\[
	\tK^\MW_*(S) \to \pi_0\Map(\1_S,(\Sigma^{-1}\P^1)^{\otimes *}).
	\]
	If $\pi_0\scr O(S)$ is a field or a local ring whose residue field has at least $4$ elements, these further relations already hold in $\K^\MW_*(S)$ by \cite[Lemma 4.4(2)]{SchlichtingEuler}, so that $\K^\MW_*(S)=\tK^\MW_*(S)$.
	\end{enumerate}
\end{remark}

\begin{remark}[Grothendieck–Witt theory]
	In forthcoming work, we define the ``genuine'' Grothendieck–\allowbreak Witt theory of qcqs derived schemes (extending \cite[Definition 3.3.3]{CHN} to non-classical schemes) and show that it is represented by an absolute motivic $\E_\infty$-ring spectrum $\mathrm{KO}$.
	Hence, if $S$ is qcqs, we have maps of commutative rings
	\[
	\begin{tikzcd}
		\tK^\MW_0(S) \ar{r} \ar{dr} & \pi_0\Map(\1_S,\1_S) \ar{d} \\
		& \pi_0\Map(\1_S,\mathrm{KO})\rlap{${}=\pi_0\GW(S)$,}
	\end{tikzcd}
	\]
	where the horizontal map was defined in Remark~\ref{rmk:KMW}(ii).
	The vertical map is given by $\A^1$-localization if $S$ is a field or a regular semilocal ring over a field of characteristic not $2$ \cite[Theorem~10.12]{norms}.
	The diagonal map sends $\langle a\rangle$ to the symmetric bilinear form $(\scr O_S,(x,y)\mapsto axy)$, and it is an isomorphism if $S$ is a field, a local ring whose residue field has at least $3$ elements \cite[Theorem~1.3]{RogersSchlichting}, or a connected semilocal ring whose residue fields have at least $7$ elements and characteristic not $2$ \cite[Theorem 2.2]{gille2015quadratic}. 
	In these cases, $\pi_0\GW(S)$ is thus a retract of $\pi_0\Map(\1_S,\1_S)$ as a ring.
\end{remark}

\section{The positive eigenspace of the rational motivic sphere}

Let $S$ be a derived scheme.
There is a splitting $\1[\tfrac 12]=\1[\tfrac 12]^+\times \1[\tfrac 12]^-$ in $\CAlg(\MS_S)$, where $\1[\tfrac 12]^{\pm}$ are the $\pm 1$-eigenspaces of the swap automorphism on $\P^1\otimes \P^1$. It induces a splitting of categories
\[
\MS_S[\tfrac 12] = \MS_S[\tfrac 12]^+ \times \MS_S[\tfrac 12]^-.
\]
By Lemma~\ref{lem:swap-GW}, the swap automorphism of $\P^1\otimes\P^1$ is $\langle -1\rangle$.
Thus, a $2$-periodic motivic spectrum $E$ is a $\1[\tfrac 12]^+$-module (resp.\ a $\1[\tfrac 12]^-$-module) if and only if $\langle -1\rangle$ acts as the identity (resp.\ as minus the identity) on $E$.

\begin{remark}\label{rmk:1/2+}
	\leavevmode
	\begin{enumerate}
		\item In $\MS_S[\tfrac 12]^+$, we have $\epsilon=-1$ and hence $\eta=0$ by~\eqref{eqn:epsilon}.
		\item If $E\in \MS_S$ is orientable, then $E[\tfrac 12]$ is a $\1[\tfrac 12]^+$-module: in fact, by the projective bundle formula, 
	any automorphism of $\scr E\in\Vect(S)$ acts as the identity on $\Th_S(\scr E)\otimes E$ \cite[\sect 6]{AHI}.
	\end{enumerate}
\end{remark}

\begin{theorem}\label{thm:Q+}
	For any derived scheme $S$, $\1_\Q^+\in \MS_S$ is canonically oriented.
\end{theorem}

\begin{corollary}\label{cor:Q+}
	For any derived scheme $S$, there is a unique isomorphism $\1_\Q^+=\HH\Q$ in $\CAlg(\MS_S)$.
\end{corollary}

\begin{proof}
	It suffices to show that $\1_\Q^+$ is an $\HH\Q$-module, since both $\1_\Q^+$ and $\HH\Q$ are idempotent algebras \cite[Theorem 9.16]{atiyah}.
	This follows from Theorem~\ref{thm:Q+} and \cite[Proposition 9.19]{atiyah}.
\end{proof}

For the proof of Theorem~\ref{thm:Q+}, we follow the strategy in Peter Arndt's thesis \cite[\sect 3.4.1]{ArndtThesis}.

\begin{lemma}\label{lem:Sym-square-zero}
	\leavevmode
	\begin{enumerate}
	\item The canonical map
	\[
	\bigoplus_{n\geq 0}\Sigma_{\P^1}^n\1\to \Sym(\P^1)
	\]
	becomes an isomorphism in $(\MS_S)_\Q^+$.
	\item The canonical map
	\[
	\1\oplus\Sigma^{-1}\P^1\to \Sym(\Sigma^{-1}\P^1)
	\]
	becomes an isomorphism in $(\MS_S)_\Q^+$.
	\end{enumerate}
\end{lemma}

\begin{proof}
	By definition, the swap automorphism of $\P^1\otimes \P^1$ is the identity in $\MS_S[\tfrac 12]^+$.
	Thus, the group $\Sigma_n$ acts trivially (up to homotopy) on $(\P^1)^{\otimes n}$, hence acts via the sign on $(\Sigma^{-1}\P^1)^{\otimes n}$.
	Rationally, it follows that $(\P^1)^{\otimes n}\to \Sym^n(\P^1)$ is an isomorphism and that $\Sym^n(\Sigma^{-1}\P^1)=0$ for $n\geq 2$.
\end{proof}

\begin{proof}[Proof of Theorem~\ref{thm:Q+}]
	Recall the retract diagram of augmented $\E_\infty$-rings
	\[
	Q\xrightarrow{\partial} \G_{\m+}\xrightarrow{r} Q
	\]
	from Construction~\ref{cst:Q}, where $Q=\1\oplus\Sigma^{-1}\P^1$.
	Note that the composite
	\[
	\1\oplus\Sigma^{-1}\P^1 \to \Sym(\Sigma^{-1}\P^1) \xrightarrow\partial \G_{\m+} \xrightarrow r Q
	\]
	is an isomorphism, where $\partial$ and $r$ are maps in $\CAlg(\MS_S)_{/\1}$. After tensoring with $\1_\Q^+$, the first map becomes an isomorphism by Lemma~\ref{lem:Sym-square-zero}(ii). Hence, $\Sym(\Sigma^{-1}\P^1)$ is a retract of $\G_{\m+}$ in $\CAlg((\MS_S)_\Q^+)_{/\1}$.
	Taking suspensions in $\CAlg((\MS_S)_\Q^+)_{/\1}$, we deduce that $\Sym(\P^1)$ is a retract of $\Pic_+$ in $\CAlg((\MS_S)_\Q^+)_{/\1}$. We thus obtain an $\E_\infty$-map
	\begin{equation}\label{eqn:Pic-to-Sym}
	\Pic_+ \to \Sym(\Sigma_{\P^1}\1_\Q^+)
	\end{equation}
	extending the canonical map $\P^1_+\to \Sym(\Sigma_{\P^1}\1_\Q^+)$. In particular, $\1_\Q^+$ is canonically oriented.
\end{proof}

\begin{remark}
	By the universal property of the K-theory of derived schemes proved by Annala and Iwasa \cite[Theorem 5.3.3]{AnnalaIwasa2}, the $\E_\infty$-map~\eqref{eqn:Pic-to-Sym} induces a morphism of motivic $\E_\infty$-ring spectra
	\[
	\KGL\to \Sym(\Sigma_{\P^1}\1_\Q^+)[u^{-1}]= \bigoplus_{n\in \Z}\Sigma^n_{\P^1}\1_\Q^+,
	\]
	where $u\colon \P^1\to \Sym(\Sigma_{\P^1}\1_\Q^+)$ is the tautological map and the equality follows from Lemma~\ref{lem:Sym-square-zero}(i).
	Under the identification $\1_\Q^+=\HH\Q$ of Corollary~\ref{cor:Q+}, this recovers the Chern character of \cite[Remark~9.20]{atiyah}.
\end{remark}

\section{Beilinson motives}
\label{sec:beilinson}

We give several characterizations of the full subcategory $\Mod_{\HH\Q}(\MS_S)\subset\MS_S$.

\begin{proposition}\label{prop:1-etale}
	If $S$ is any derived scheme, then $\1^\et[\tfrac 12]^-=0$ in $\MS_S^\et$. Hence, $\1^\et_\Q=\HH\Q$.
\end{proposition}

\begin{proof}
	It suffices to prove this over $\Z[i]$ and $\Z[\zeta_3]$, since these generate an étale covering sieve of $\Z$.
	In both rings, $-1$ is a unit sum of squares, so the result follows from Proposition~\ref{prop:-1}.
	The last statement follows from Corollary~\ref{cor:Q+} and the fact that $\HH\Q$ is an étale sheaf.
\end{proof}

\begin{proposition}\label{prop:HQ-etale}
	Every $\HH\Q$-module in $\MS_S$ satisfies étale descent.
\end{proposition}

\begin{proof}
	It suffices to check that the sieve generated by a surjective finite étale map $f\colon T\to S$ becomes an isomorphism in $\Mod_{\HH\Q}(\MS_S)$, and we may assume $S$ affine. By Atiyah duality \cite[Theorem 5.9]{atiyah}, the canonical map $\alpha_f\colon f_\sharp \to f_*$ is an isomorphism.
	Let $\tau_f$ be the endomorphism of $\id_{\MS_S}$ given by
	\[
	\id\xrightarrow{\smash[t]{\mathrm{unit}}} f_*f^* \xleftarrow{\alpha_f} f_\sharp f^*\xrightarrow{\smash[t]{\mathrm{counit}}} \id.
	\]
	By \cite[Proposition 5.12(ii)]{atiyah}, if $E$ is an $R$-module, then $\tau_f\colon E\to E$ is $R$-linear and is multiplication by $\tau_f\colon R\to R$ (hence by $\tau_f\colon \1\to \1$).
	For $R=\HH\Z$, we claim that $\tau_f$ is multiplication by the degree of $f$.
	This is known if $S$ is smooth over $\Z$, and we can reduce to this case since the stack of finite étale schemes is left Kan extended from smooth $\Z$-algebras and $\tau_f$ is compatible with base change \cite[Proposition 5.12(i)]{atiyah}. Thus, since the degree of $f$ is everywhere positive, the identity of $\Mod_{\HH\Q}(\MS_S)$ is a retract of $f_*f^*$, so $f^*$ is conservative. But $f^*$ sends the sieve generated by $f$ to an isomorphism, so we are done.
\end{proof}

\begin{corollary}\label{cor:HQ-etale}
	Let $S$ be a derived scheme. Then $\Mod_{\HH\Q}(\MS_S)=(\MS_S^\et)_\Q$.
\end{corollary}

\begin{proof}
	Combine Propositions \ref{prop:1-etale} and~\ref{prop:HQ-etale}.
\end{proof}

\begin{theorem}
	\label{thm:beilinson}
	Let $S$ be a derived scheme and let $E\in(\MS_S)_\Q$ be a rational motivic spectrum over $S$. The following conditions are equivalent:
	\begin{enumerate}
		\item $E$ is an $\HH\Q$-module.
		\item $E$ admits a structure of $\MGL$-module.
		\item $E$ is orientable.
		\item The swap automorphism of $\Sigma_{\P^1}^2 E$ is the identity.
		\item The involution $\langle -1\rangle\colon E\to E$ is the identity.
		\item The Hopf map $\eta\colon \Sigma_{\P^1} E\to \Sigma E$ is zero.
		\item $E$ satisfies étale descent.
	\end{enumerate}
	Moreover, these conditions always hold if $-1$ is a unit sum of squares in each henselian local ring of $S$ (see Definition~\ref{def:unit-sum} and Remark~\ref{rmk:unit-sum}).
\end{theorem}

\begin{proof}
	The equivalence (i)~$\Leftrightarrow$~(ii) was proved in \cite[Proposition 9.19]{atiyah}.
	Since $\MGL$ is an orientable $\E_\infty$-ring \cite[\sect 7]{AHI}, any $\MGL$-module is orientable, hence (ii)~$\Rightarrow$~(iii).
	The equivalence (i)~$\Leftrightarrow$~(iv) is Corollary~\ref{cor:Q+} and the equivalence (i)~$\Leftrightarrow$~(vii) is Corollary~\ref{cor:HQ-etale}.
	The equivalence (iv)~$\Leftrightarrow$~(v) follows from Lemma ~\ref{lem:swap-GW} and the implication (vi)~$\Rightarrow$~(v) from Lemma~\ref{lem:brackets}.
	The implications (iii)~$\Rightarrow$~(iv)~$\Rightarrow$~(vi) follow from Remark~\ref{rmk:1/2+} and complete the circle.
	Finally, the last statement follows from Proposition~\ref{prop:-1}.
\end{proof}

\begin{remark}
	In Theorem~\ref{thm:beilinson}, the implications (ii) $\Rightarrow$ (iii) $\Rightarrow$ (vi) $\Rightarrow$ (iv) $\Leftrightarrow$ (v) hold for any $E\in \MS_S$ (the second by Proposition~\ref{prop:hopf}), and the implications (vii) $\Rightarrow$ (v) $\Rightarrow$ (vi) hold for any $E\in\MS_S[\tfrac 12]$.
	
	We note that the Steinberg relation only enters the proof of Theorem~\ref{thm:beilinson} when dealing with $\Z[\zeta_3]$ in Proposition~\ref{prop:1-etale}. Thus, it is not needed to prove the equivalence of (i)–(vi), nor to prove the equivalence with (vii) if $S$ is a $\Z[\tfrac 12]$-scheme.
\end{remark}

\providecommand{\bysame}{\leavevmode\hbox to3em{\hrulefill}\thinspace}
\providecommand{\href}[2]{#2}

\end{document}